\documentclass{article}
\usepackage{amstext,amsfonts,amsmath,amsthm,amssymb,amscd}

\newtheorem{theorem}{Theorem}
\newtheorem{lemma}{Lemma}
\newtheorem{question}{Question}

\begin{document}
\title{ Pencil of irreducible rational curves and\\ Plane Jacobian conjecture}
\author{Nguyen Van Chau\thanks{Supported in part by NAFOSTED, Vietnam
and Tokyo University of Sciences, Tokyo, Japan.}}

\date{}

\maketitle

\begin{abstract}
We are concerned with the behavior of the polynomial maps
$F=(P,Q)$ of $\mathbb{C}^2$ with finite fibres and satisfying the
condition that all of the curves $aP+bQ=0$, $(a:b)\in
\mathbb{P}^1$, are irreducible rational curves. The obtained
result shows that such polynomial maps $F$ is invertible if
$(0,0)$ is a regular value of $F$ or if the Jacobian condition
holds.



\end{abstract}
\noindent {\bf 1.} The mysterious  Jacobian Conjecture (see
\cite{Bass} and \cite{Essen-book} for its history and surveys),
posed first by Ott-Heinrich Keller \cite{Keller} since 1939 and
remains open even for the case $n=2$, asserts that every
polynomial map $F$ of $\mathbb{C}^n$  satisfying the Jacobian
condition $\det DF\equiv const.\neq 0$ is invertible, and hence,
is a polynomial automorphism of $\mathbb{C}^n$. The following
results, which appeared in the literature in some convenient
statements, characterize the invertibility of non-zero constant
Jacobian polynomial maps $F$ in terms of the topology of inverse
images $F^{-1}(l)$ of the complex lines $l\subset \mathbb{C}^n$,

\begin{theorem} Let $F$ be a polynomial
 map of \ $\mathbb{C}^n$ with non-zero constant Jacobian, $\det DF\equiv const.\neq 0
$. Then,
 \begin{enumerate}

\item[i)] $F$ is invertible if the inverse images $F^{-1}(l)$ of
complex lines $l\subset\mathbb{C}^n$ having same a fixed direction
are irreducible rational curves, and

\item[ii)] $F$ is invertible if  for generic point $q\in
\mathbb{C}^n$ the inverse images $F^{-1}(l)$ of complex lines
$l\subset\mathbb{C}^n$ passing through $q$ are irreducible
rational curves.

\end{enumerate}
\end{theorem}

Here, we mean an irreducible rational curve to be an algebraic
curve homeomorphic to the 2-dimensional sphere with a finite
number of  punctures.

Theorem 1 (ii), due to Nollet and Xavier (Corollary 1.3,
\cite{NF}), is deduced from a deep result  on the holomorphic
injectivity (Theorem 1.1, in \cite{NF}). Theorem 1 (i) appears
earlier with algebraic and algebra-geometric proofs in Razar
\cite{Razas}, Le and Weber \cite{LeWe2}, Friedland
\cite{Friedland},  and Heitmann \cite{Heitmann}  for $n=2$, and in
Nemethi and Sigray \cite{NemethiSigray} for general case. In fact,
as observed by Vistoli \cite{Vistoli} and  by Neumann and Norbudy
\cite{NeumannNorbudy}, {\it non-trivial rational polynomials in
two variable must have reducible fibres}.

In this short article we would like to note that in certain cases
the invertibility of polynomial map $F=(P,Q)$ of $\mathbb{C}^2$
with finite fibres can be characterized by the irreducibility and
rationality of the curves $aP+bQ=0$, $(a:b)\in \mathbb{P}^1$. Our
result is
\begin{theorem}[Main Theorem]\label{Main}  Let $F=(P,Q)$ be a
polynomial map of\ $\mathbb{C}^2$ with finite fibres such that all
of the curves $aP+bQ=0$, $(a:b)\in \mathbb{P}^1$, are irreducible
and rational. Then, the followings are equivalent
\begin{enumerate} \item[a)] $(0,0)$ is a regular value of $F$;
\item[b)] $\det DF\equiv const.\neq 0$; \item[c)] $F$ is
invertible.
\end{enumerate}
\end{theorem}

Theorem \ref{Main} leads to a little surprise that for the case
$n=2$ Theorem 1 (ii) is still valid without  the Jacobian
condition.

\begin{theorem}\label{Cri} Let $F$ be a  polynomial map of\ $\mathbb{C}^2$ with
finite fibres. If  for generic points $q\in \mathbb{C}^2$ the
inverse images $F^{-1}(l)$ of complex lines $l\subset\mathbb{C}^2$
passing through $q$ are irreducible rational curves, then $F$ is
invertible.
\end{theorem}
\begin{proof} Since the fibres
of $F$ are finite, we have $\det DF\not \equiv 0$. Then, by the
assumptions we can assume that  $(0,0)$ is a regular value of $F$
and for all lines $l$ passing through $(0,0)$ the inverse images
$F^{-1}(l)$ are irreducible rational curves. Hence, by Theorem
\ref{Main} the map $F$ is invertible.
\end{proof}
In attempt to understand the plane Jacobian conjecture it is worth
to consider the questions:

\begin{question} Does  the Jacobian condition ensure the irreducibility of all of the
curves $aP+bQ=0$, $(a:b)\in \mathbb{P}^1$ ?
\end{question}

\begin{question} Is a non-zero constant Jacobian polynomial map $F=(P,Q)$
of $\mathbb{C}^2$ invertible if  all of the curves $aP+bQ=0$,
$(a:b)\in \mathbb{P}^1$, are irreducible ?
\end{question}
Kaliman \cite{Kaliman93} observe that to prove the plane Jacobian
conjecture it is sufficient to consider non-zero constant Jacobian
polynomial maps $F=(P,Q)$, in which all of  fibres $P=c$, $c\in
\mathbb{C}$, are irreducible. Relating to Question 2 note that the
only irreducibility of the curves $aP+bQ=0$, $(a:b)\in
\mathbb{P}^1$, does not not guaranty the invertibility of the
polynomial map $F=(P,Q)$. For example, the map
$F(x,y)=(x,x^2+y^3)$ is not invertible, but the curves
$ax+b(x^2+y^3)=0$, $(a:b)\in \mathbb{P}^1$, are irreducible.
Further deep examinations on the relation between the Jacobian
condition and the geometry of the pencil of curves $aP+bQ=0$ would
be useful in the pursuit for the solution of the plane Jacobian
problem.

The proof of Theorem \ref{Main} will be carried out in \$3 after
some necessary preparations in \$ 2.

\vskip0,5cm \noindent{\bf 2.} From now on $F=(P,Q)$, is a given
polynomial map of $\mathbb{C}^2$ with finite fibres. Our proof of
Theorem \ref{Main} is based on the facts below.

i) Following \cite{Jelonek}, by the non-propoer value set $A_F$ of
$F$  we mean the set of all values $a\in \mathbb{C}^2$ such that
$a$ is the limit set of $F(v_k)$ for a sequence $v_k\in
\mathbb{C}^2$ tending to $\infty$. The set $A_F$ is a plane curve
composed of the images of some polynomial maps from $\mathbb{C}$
into $\mathbb{C}^2$ \cite{Jelonek}. When $F$ has finite fibres, by
definitions
$$ v\not\in  A_F \Leftrightarrow \sum_{w\in F^{-1}(v)} \deg_wF=
\deg_{geo.} F, \eqno(1)$$ where $\deg_w F$ is the multiplicity of
$F$ at $w$ and $\deg_{geo.}F$ is the number of solutions of the
equation $F=v$ for generic points $v\in \mathbb{C}^2$. In the case
when $F$ satisfies the Jacobian condition we have

\begin{theorem}[\cite{Chau99}, \cite{Chau04}] Let $F=(P,Q)$ be a non-zero constant
Jacobian polynomial map. Then, the irreducible components of
$A_F$, if exists, can be parameterized by polynomial maps $
t\mapsto (\varphi(t),\psi(t))$, $\varphi, \psi\in \mathbb{C}[t]$,
satisfying
$${\deg \varphi\over\deg \psi}={\deg P\over \deg Q}.$$ In
particular, $A_F$ can never contains components isomorphic to the
line $\mathbb{C}$.
 \end{theorem}

ii) Let  $D_\lambda:=\{(x,y)\in \mathbb{C}^2: aP(x,y)+bQ(x,y)=0\}$
for $\lambda=(a:b)\in \mathbb{P}^1$ and denote by $r_\lambda$ the
number of irreducible components of the curve $D_\lambda$.
Regarding the plane $\mathbb{C}^2$ as a subset of the projective
plane $\mathbb{P}^2$, we can associate to $F$ the rational map
$G:\mathbb{P}^2\longrightarrow \mathbb{P}^1$ given  by
$G(x,y)=(P(x,y):Q(x,y))\in \mathbb{P}^1,$ which is well defined
outside the finite set $B:=F^{-1}(0,0)$ and a possible finite
subset of the line at infinity of $\mathbb{C}^2$. We can extend
$G$ to a regular morphism $g: X\longrightarrow \mathbb{P}^1$ from
a compactification $X$ of $\mathbb{C}^2\setminus B$ to
$\mathbb{P}^1$. By a {\it horizontal component} ({\it constant
component}) of $G$ we mean an irreducible component $\ell$ of the
divisor $\mathcal{D}:=X\setminus (\mathbb{C}^2\setminus B)$ such
that the restriction $g_l$ of $g$ to $l$ is a non-constant mapping
(res. constant mapping). Let us denote by $h_G$ the number of
horizontal components $l$ of $g$. The number $h_G$ is depended on
$P$ and $Q$, but not on the compactification $X$ of
$\mathbb{C}^2$.

We can construct such extension $g: X\longrightarrow \mathbb{P}^1$
by a minimal sequence of the blowing-ups $\pi: X\longrightarrow
\mathbb{P}^2$ that removes all of the indeterminacy points of the
rational map $G$. In such an extension $g$ the divisor
$\mathcal{D}$ is the disjoint union of the connected divisors
$\mathcal{D}_\infty:=\pi^{-1}(L_\infty)$ and
$\mathcal{D}_b:=\pi^{-1}(b)$, $b\in B$, where $L_\infty$ indicates
the line at infinity of $\mathbb{C}^2\subset \mathbb{P}^2$.
Denotes by $h_\infty$ and $h_b$ the numbers of horizontal
components of $G$ contained in the divisors $\mathcal{D}_\infty$
and $\mathcal{D}_b$, $b\in B$, respectively. Obviously,
$$h_\infty >0\text{ and } h_b>0 \text{ for } b\in B\eqno(2)$$
and
$$h_G=h_\infty+\sum_{b\in B}h_b.\eqno(3)$$
\begin{lemma}\label{Lem1} If the generic curve $D_\lambda$
is  irreducible and rational, then
$$\sum_{\lambda\in \mathbb{P}^1} (r_\lambda -1)=h_\infty+\sum_{b\in B}h_b-2.\eqno(4)$$
\end{lemma}
The equality (4) is a folklore fact which can be reduced from the
estimation on the total reducibility order of pencils of curves
obtained by Vistoli in \cite{Vistoli}. The proof presented below
is quite elementary and is analogous to those of Kaliman
\cite{Kaliman92} for the total reducibility order of polynomials
in two variables.

\begin{proof}[Proof of Lemma \ref{Lem1}] Fixed a regular morphism $g$ which is a  blowing-up version of $G$.
Let $C_\lambda$ be the fiber $g=\lambda$, $\lambda\in
\mathbb{P}^1$, and let $C$ be a generic fiber of $g$. We will use
Suzuki's formula \cite{Suzuki}
$$\sum_{\lambda\in \mathbb{P}^1}(\chi(C_\lambda)-\chi(C))=\chi(X)-2\chi(C).\eqno(5)$$
Here,  $\chi(V)$ indicates the Euler-Poincare characteristic of
$V$.

Let us denote by $m$  the number of irreducible components of the
divisor $\mathcal{D}$ and by $m_\lambda$ the number of irreducible
components of $C_\lambda$ contained in $\mathcal{D}$. Then, we
have $\chi(X)=m+2$ and
$$m=h_\infty+\sum_{b\in B}h_b+\sum_{\lambda\in \mathbb{P}^1}m_\lambda.$$

Since the generic curves $D_\lambda$ are irreducible and rational
 the generic fibre $C$ of $g$ is a copy $\mathbb{P}^1$
and the fibres $C_\lambda$ are connected rational curves with
simple normal crossing. Therefore, $\chi(C)=2$ and
$\chi(C_\lambda)=r_\lambda+m_\lambda+1.$

Now, by the above estimations we have
$$\chi(X)-2\chi(C)=h_\infty+\sum_{b\in B}h_b+\sum_{\lambda\in
\mathbb{P}^1}m_\lambda-2\eqno(6)$$ and
$$\sum_{\lambda\in \mathbb{P}^1}(\chi(C_\lambda)-\chi(C))=\sum_{\lambda\in
\mathbb{P}^1}(r_\lambda-1) +\sum_{\lambda\in
\mathbb{P}^1}m_\lambda.\eqno(7)$$ Putting (6) and (7)  into (5) we
get the desired equality (4).
\end{proof}

iii) Regarding polynomials $P$ and $Q$ as rational maps from
$\mathbb{P}^2$ into $\mathbb{P}^1$, the blowing-up
$X\stackrel{\pi}{\longrightarrow} \mathbb{P}^2$ in (ii) also
provides natural extensions $p,q:X\longrightarrow \mathbb{P}^1$ of
$P$ and $Q$, which may have some indeterminacy points. If
necessary, we can replace $X$ by its convenient blowing-up version
so that $p$ and $q$ are regular morphisms and
$f=(p,q):X\longrightarrow \mathbb{P}^1\times\mathbb{P}^1$ is a
regular extension of $F$.

The restrictions of $p$ and $q$ to each irreducible component
$l\subset \mathcal{D}$ then determine holomorphic maps from $l$ to
$\mathbb{P}^1$, denoted by $p_l$ and $q_l$ respectively. We can
divide horizontal components $l$ of $G$ into some following types:
\begin{enumerate}
\item[I)] $l\subset \mathcal{D}_b$. Then, $(p_l,q_l)\equiv(0,0)$

\item[II)] $l\subset \mathcal{D}_\infty$. Then, either

  a) $(p_l,q_l)\equiv(\infty,\infty)$,

  b) $(p_l,q_l)\equiv(0,0)$, or

  c) $(p_l,q_l)$ is a non-constant mapping with
  $(p_l:q_l)\neq const.$ .
\end{enumerate}

Obvious, in Type (IIc) $(p_l,q_l)(l)\cap \mathbb{C}^2\neq
\emptyset$.

By {\it dicritical component} of $F$ we mean an irreducible
component $l\subset \mathcal{D}_\infty$ such that $(p_l,q_l)$ is a
non-constant mapping. Obviously, by the definitions
$$A_F=\bigcup_{l \text{ dicritical components of } F} (f(l)\cap
\mathbb{C}^2).$$ In particular, $F$ is a proper map of
$\mathbb{C}^2$ if and only if $F$ does not have  dicritical
components.

\begin{lemma}\label{Lem2} We have
\begin{enumerate}
\item[a)] $G$ has at least one horizontal component of Type {\rm
(IIa)}; \item[b)] If $A_F\neq \emptyset$, then $G$ has at least
one horizontal component of Types {\rm (IIb)} or Type {\rm (IIc)}.
If $(0,0)\in A_{F}$, then $G$ has at least one horizontal
component of Type {\rm (IIb)}; \item[c)] If $l$ is a dicritical
component of $F$, then either $l$ is a horizontal component of $G$
or $f(l)\cap\mathbb{C}^2$ is a line passing through $(0,0)$.
\end{enumerate}

\end{lemma}
\begin{proof} a) Note that each generic fiber $C_\lambda$ is the union of $D_\lambda$ and a finite
number of points lying in horizontal components of $f$, at which
the rational map $(p,q)$ is well defined. If $f$ does not have
horizontal components of Type (IIa), the map $(p,q)$ would obtains
finite values on $C_\lambda\cap \mathcal{D}$, and hence, $P$ and
$Q$ would be constant on each connected component of $D_\lambda$.
This is impossible, since  the fibres of $F$ are finite.

b) By definitions the non-proper value set $A_F$ can be expressed
as $A_F=f(\mathcal{D}_\infty)\cap \mathbb{C}^2.$ Assume $A_F\neq
\emptyset$. Let  $V$ be an irreducible
 component of $A_F$. Then, the inverse $f^{-1}(V)$ must contains
 a  component $l$ of $\mathcal{D}_\infty$
 such that $V\subset f(l)$. Obviously, $g(l)=\mathbb{P}^1$ or ${g}_l\equiv const.$.
 Therefore, $l$ is a horizontal component of Type (IIc) of $G$
 , except when $(0,0)\in A_F$ and $V$ is a line passing through $(0,0)$.
 In the case $(0,0)\in A_F$, the intersection
 $D:=f^{-1}(0,0)\cap \mathcal{D}_\infty$ is not empty.
 Then, $f$ maps each neighborhood $U$ of $D$ onto a
neighborhood of $(0,0)$, and hence, ${g}$ maps such neighborhood
$U$ onto $\mathbb{P}^1$. It follows that $D$ must contains a
horizontal component of Type (IIb) of $G$. The conclusions now are
clear.

c) Let $l$ be a dicritical component of $F$, $(p_l,q_l)\neq
const.$ . By definitions, $l$ is either a horizontal component of
$G$ if $(p_l:q_l)\neq const.$ or a component of a fiber of $g$.
Obviously, in the late case $f(l)\cap \mathbb{C}^2$ is a line
passing through $(0,0)$.

\end{proof}

\vskip 0,5 cm \noindent{\bf 3.} Now, we are ready to prove Theorem
\ref{Main}.

\begin{proof}[Proof of Theorem \ref{Main}]

Let $F=(P,Q)$ be a given polynomial map $\mathbb{C}^2$ with finite
fibres satisfying that all of the curves $aP+bQ=0$, $(a:b)\in
\mathbb{P}^1$, are irreducible and rational. The implication
$(c)\Rightarrow (a;b)$ is trivial. We need to prove only
$(a)\Rightarrow (c)$ and $(b)\Rightarrow (c)$. We will use same
constructions and notations presented for $F=(P,Q)$ in the
previous sections.

First, by assumptions we can apply Lemma \ref{Lem1} to see that
$G$ has exactly two horizontal components,
$$h_G=h_\infty+\sum_{b\in B}h_b=2.\eqno(8)$$
Since $h_\infty >0$  and $h_b>0$ for $b\in B$ by Lemma \ref{Lem2},
from (8) it follows that either
\begin{enumerate}
\item [i)]$h_\infty =2$ and $B=\emptyset$,
 or
\item[ii)]$h_\infty =1$, $B$ consists of an unique point, say
$B=\{ b\}$, and $h_b=1$.
\end{enumerate}

$(a)\Rightarrow (c)$. Assume that $(0,0)$ is a regular value of
$F$, i.e $F^{-1}(0,0)$ is non-empty and does not contain singular
points of $F$. So, we drop into Situation (ii): $h_\infty=1$,
$B=\{b\}$ and $h_b=1$. Then, by Lemma \ref{Lem2} (a) the unique
horizontal component of $G$ in $\mathcal{D}_\infty$ must be in
Type (IIa). It follows that $F$ does not have  dicritical
component, $A_F=\emptyset$. This means that $F$ is a proper map of
$\mathbb{C}^2$, or equivalent $A_F=\emptyset$. Then, by (4) the
geometric degree $\deg_{geo.}F$ of $F$ is equal to the number of
solutions of the equation $F(x,y)=0$, counted with multiplicity.
But, this equation accepts $b$ as an unique solution and $b$ is
not singular point of $F$. Thus,  $\deg_{geo.}F=1$ and hence $F$
is injective. Then, by the well-known fact (see \cite{Essen-book})
that polynomial injections of $\mathbb{C}^n$ are automorphisms the
map $F$ must be invertible.

$(b)\Rightarrow (c)$.  Assume $\det DF\equiv const. \neq 0$. If
$F^{-1}(0,0)\neq \emptyset$, the value $(0,0)$ then is a regular
value of $F$ and we are done by the previous part. Assume the
contrary that $F^{-1}(0,0)=\emptyset$. Then, we drop into
Situation (i): $h_\infty=2$ and $B=\emptyset$. In this case, by
definitions $(0,0)$ is a non-proper value of $F$, $(0,0)\in A_F$.
Therefore, by Lemma \ref{Lem2} (a) and (b) $G$ has  exactly two
horizontal components, one is of Type (IIa) and one is of Type
(IIb). In particular, none of such horizontal components can be a
dicritical component of $F$. Hence, by Lemma \ref{Lem2} (c) $A_F$
must be composed of some lines passing through $(0,0)$. This
contradicts to Theorem 4. Thus, $F$ is invertible.

\end{proof}

\bigskip

\noindent{\it Acknowledgments}:  The author wishes to thank Prof.
Mutsuo Oka  for his helps and useful discussions when he visit
Tokyo University of Sciences in 2010.

\noindent INSTITUTE OF MATHEMATICS, 18 HOANG QUOC VIET, 10307,
HANOI, VIETNAM. E-mail: nvchau@math.ac.vn .

\end{document}